\documentclass[12pt,reqno]{amsart}
\usepackage{amsmath, amsthm, amssymb, bbm, mathtools, nicefrac}
\usepackage[main=english,brazil]{babel}
\usepackage[utf8]{inputenc}
\usepackage[pdftex,dvipsnames]{xcolor}  
\usepackage{hyperref}
\usepackage{url}
\usepackage{fixme}
\usepackage{multicol}
\usepackage{xspace}
\usepackage[all]{xy}
\usepackage{enumitem}
\usepackage{scalefnt}
\usepackage{aligned-overset}

\usepackage{a4wide}

\newcommand*{\B}{\mathcal{B}}

\newcommand{\id}{\operatorname{id}} 

\DeclareMathOperator{\Hom}{Hom}
\DeclareMathOperator{\Ext}{Ext}
\DeclareMathOperator{\Inn}{Inn}
\DeclareMathOperator{\Aut}{Aut}
\DeclareMathOperator{\Out}{Out}
\DeclareMathOperator{\Pic}{Pic}
\DeclareMathOperator{\spann}{span}

\newcommand{\cf}{\mbox{cf.}\xspace}				
\newcommand{\eg}{\mbox{e.\,g.}\xspace}			
\newcommand*{\ie}{\mbox{i.\,e.}\xspace}			
\newcommand*{\etal}{\mbox{et\,al.}\xspace}		
\newcommand{\wrt}{\mbox{w.\,r.\,t.}\xspace}		
\newcommand{\aka}{\mbox{a.\,k.\,a.}\xspace}		
\newcommand*{\ndash}{\nobreakdash-}
\newcommand*{\Star}{$^*$\ndash}
\newcommand*{\acts}{\,.\,}

\theoremstyle{plain}
	\newtheorem{teo}{Theorem}[section]
	\newtheorem{lemma}[teo]{Lemma}
	\newtheorem{cor}[teo]{Corollary}

\theoremstyle{definition}
	\newtheorem{defi}[teo]{Definition}
	\newtheorem{rem}[teo]{Remark}
	\newtheorem{example}[teo]{Example}

\hypersetup{colorlinks=true,
  linkcolor = black,
  citecolor = blue,
  urlcolor =  blue, 
}

\begin{document}

\title[Saturated Fell bundles and their classification]{Saturated Fell bundles and their classification}
\date{\today}

\author{Natã Machado}
\address{Universidade Federal de Santa Catarina}
\email{nmachado.ufsc@gmail.com}

\author{Stefan Wagner}
\address{Blekinge Tekniska H\"ogskola}
\email{stefan.wagner@bth.se}



\begin{abstract}
We present a comprehensive classification theory for saturated Fell bundles over locally compact groups, utilizing data associated with their base group and unit fiber. 
This framework offers a unified approach to understanding the structure and properties of such bundles, providing key insights into their classification.
\end{abstract}

\maketitle

\section{Introduction}

\subsubsection*{Fell bundles}

To unravel the structure of a given algebra, a common approach is to consider a grading on the algebra.
This provides a decomposition into homogeneous components, which are subspaces of the algebra itself but can also be studied as distinct entities. 
To put it differently, starting from a graded algebra, one can view its homogeneous components as the parts remaining after
disassembling the algebra along its grading.
Examining these components individually can provide valuable insights into the algebra's structure, revealing aspects that might otherwise remain hidden.

In this paper, we explore the concept of \emph{Fell bundles}, introduced by Fell in his foundational works \cite{Fell1, Fell2} under the name of \emph{C\Star algebraic bundles} to deal with disassembled C\Star algebras.
These objects play a central role in operator algebras, noncommutative geometry, and mathematical physics, providing a powerful framework for studying group representations as well as C\Star dynamical systems and their associated crossed~products.
Fell bundles have attracted significant interest in recent years due to their wide range of applications and have been extensively studied by many authors (see, \eg, \cite{Bi24,BuMA23,BuEch24,DuLi25,Ex96,Ex02,Ex17,kr23,Ku98,Qu96,Rae16,Rae18,Res18,SiWi13} and ref.~therein).  
For instance, they serve as a bridge between diverse areas such as quantum logic, representation theory, and noncommutative geometry.

Roughly speaking, a Fell bundle $\B$ over a locally compact group $G$ is a bundle over $G$ such that the fiber $B_e$ over the identity element $e$ of $G$ is a C\Star algebra, known as the \emph{unit fiber}, and such that the fiber $B_g$ over an element $g \in G$ is a Hilbert $B_e$-bimodule.
Additionally, $\B$ is equipped with continuous maps, a multiplication and an involution, that reflect the group operations: $B_g \cdot B_h \subseteq B_{gh}$, for $g,h \in G$, and $B^*_g\subseteq B_{g^{-1}}$, for $g \in G$.
Notably, this definition extends naturally to Fell bundles over groupoids (see, \eg, \cite{Ku98}). 

Fell bundles are not merely abstract constructs; they provide a unifying perspective on many important examples of C\Star algebras that have been the subject of intense study over the past decades. 
A particularly illuminating example is the quantum 2-torus $\mathbb{T}^2_\theta$, $\theta \in \mathbb{R}$, which can be thought of as a Fell bundle over $\mathbb{Z}^2$.
Among the examples in which the Fell bundle structure is not so obvious, lie some of the most studied C\Star algebras of the past couple of decades. 
These include:
\begin{enumerate} 
\item 
    group C\Star algebras,
\item 
    quantum SU(2), 
\item 
    noncommutative Heisenberg manifolds,
\item 
    AF-algebras,
\item 
    Cuntz-Krieger algebras,
\item 
    Graph C\Star algebras,
\item 
    and many others.
\end{enumerate}

\subsubsection*{Relation to noncommutative principal bundles}

The tremendous work of Hopf, Stiefel, and Whitney in the 1930’s demonstrated the importance of principal bundles for various applications of geometry and mathematical physics.
In the noncommutative setting the notion of a free action of a quantum group on a C\Star algebra provides a natural framework for noncommutative principal bundles (see, \eg, \cite{BaCoHa15,Ell00,SchWa17} and ref.~therein). 
Their structure theory 
and their relation to $KK$-theory 
certainly appeal to operator algebraists and functional analysts. 
Also, just like their classical counterparts, noncommutative principal bundles are becoming increasingly prevalent in various applications of geometry and mathematical~physics. 

The study and classification of actions on C\Star algebras is a central theme in noncommutative geometry, and the experience with the commutative case suggests that free group actions are due to the lack of degeneracies easier to understand and to classify than general actions.
For instance, free actions of a compact group correspond to principal bundles over the quotient space, with locally trivial bundles being completely classified by cocycles valued in the group through Čech cohomology.

In the noncommutative setting, substantial progress has been made for compact quantum groups:
Free and ergodic actions were classified by Olesen, Pedersen, Takesaki~\cite{OlPeTa80}, and, independently, by Albeverio–Høegh-Krohn~\cite{AlHo80} for compact Abelian groups, and by Wassermann~\cite{Wass88a, Wass88b, Wass89} for compact non-Abelian groups, linking free actions to 2-cocycles of the dual group. 
An analogous result for compact quantum groups was obtained by Bichon, De Rijdt, and Vaes~\cite{BiRiVa06}.

Extending these results beyond the ergodic case becomes significantly more challenging, because, even in the case of a commutative fixed point algebra, the action cannot necessarily be decomposed into a bundle of ergodic actions. 
However, in a series of papers~\cite{SchWa15,SchWa16,SchWa17}, the first author and collaborators achieved a complete classification of free, though not necessarily ergodic, actions of compact quantum groups on unital C\Star algebras.




In the case of a free action of a group C\Star algebra of a discrete group, there exists a very close relationship with \emph{saturated Fell bundles} (see, \eg,~\cite{Ex96,Ex15,Qu96,Rae18} and ref.~therein). 
This relationship also suggests a way of addressing noncommutative principal bundles in terms of saturated Fell bundles on locally compact groups or, more generally, on groupoids.

\subsubsection*{Aim of this paper}

In this paper, we focus on the classification of saturated Fell bundles, a subclass characterized by a well-behaved and rich symmetry in their bimodule structure (see~Definition~\ref{def:fb}). 
Building on recent advancements, we aim to establish new methods and tools for understanding their structure and classification. 
Noteworthily, the approach we adopt is, to some extent, based on the framework introduced in~\cite{SchWa15}.

\subsubsection*{Organization of the paper}

In Section~\ref{sec:preliminaries}, we set out the necessary definitions, notations, and results that will be referenced throughout the paper.

In Section~\ref{sec:construction}, we present the first key step in our classification procedure, which unfolds as follows: Let $G$ be a locally compact group and let $B$ be a unital C\Star algebra, and let $\psi:G \to \Pic (B)$ be a group homomorphism, which serves as an invariant (see Lemma~\ref{lem:pic}).
While $\psi$ does not, in general, distinguish all saturated Fell bundles over $G$ with unit fiber $B$, nor does it necessarily correspond to one, refining this invariant leads to our classifying data, which we refer to as a \emph{factor system} (see Definition~\ref{def:fs}).
As a key result, we establish that every factor system indeed gives rise to a saturated Fell bundles over $G$ with unit fiber $B$ and Picard homomorphism $\psi$ (see Theorem~\ref{teo:fs} and Corollary~\ref{cor:fs}). 
The construction of this Fell bundle constitutes the main focus of Section~\ref{sec:construction}.

In Section~\ref{sec:classification}, we present the second key step in our classification procedure, detailed as follows:
Once again, let $G$ be a locally compact group, let $B$ a unital C\Star algebra, and let $\psi:G \to \Pic(B)$ be a group homomorphism. 
Initially, in Section~\ref{sec:classification_alg}, we treat $G$ as a discrete group and classify discrete saturated Fell bundles over $G$ with unit fiber $B$ and Picard homomorphism $\psi$. 
The main result establishes that, if such Fell bundles exist, then they are parametrized - up to 2-coboundaries - by \emph{2-cocycles on $G$ with values in $UZ(B)$}, the group of unitary elements in the center of $B$ (see Corollary~\ref{cor:class}).
We further provide a group-theoretic criterion for the existence of saturated Fell bundles over $G$ with unit fiber $B$ and Picard homomorphism $\psi$ (see Theorem~\ref{teo:H^3}). 
Finally, in Section~\ref{sec:classification_top}, we extend our findings to the topological setting (Corollary~\ref{cor:class_top}).
Our strategy therefore involves fixing a topology and classifying Fell bundles whose underlying Banach bundle structures are equivalent with respect to it.

In Section~\ref{sec:out}, we explore a simple yet interesting class of Fell bundles, which we refer to as \emph{crossed product bundles},
characterized by the existence of a unitary element in every fiber. 
In particular, we illustrate the construction of some of these Fell bundles using classical factor systems of groups (Theorem~\ref{teo:cp}).


\subsubsection*{Future directions}

We wish to mention that, with further technical refinement, the results presented here can be generalized to saturated Fell bundles over groupoids. 
Additionally, general Fell bundles may be classified using functors similar to those introduced in \cite{Ne13}.

Furthermore, we would like to point out that this research forms part of a broader program aimed at systematically studying geometric aspects of Fell bundles - an area that, despite the widespread use and versatility of Fell bundles, remains largely unexplored.
Future work will focus on advancing the noncommutative geometry aspects of Fell bundles through the application of spectral triples, which encode geometric phenomena within an analytic framework.

\section{Preliminaries}
\label{sec:preliminaries}

Our study concerns the classification of saturated Fell bundles by means of data associated with their base group and unit fiber as well as cohomological features. 
In what follows, we provide definitions, notations, and results that will be used repeatedly throughout the text.

We begin by providing some standard references. 
For a comprehensive treatment of operator algebras, we refer to the excellent works of Blackadar~\cite{BB06} and Pedersen~\cite{Ped18}. 
Our approach also deals with imprimitivity bimodules, \aka Morita equivalence bimodules; for an up-to-date account, see Raeburn and Williams~\cite{Rae98} and the monograph by Echterhoff \etal~\cite{EKQR06}. 
For the theory of Fell bundles, our primary references are the seminal volumes by Doran and Fell~\cite{Fell1,Fell2}.
For background on abstract group cohomology, particularly  cochains, cocycles, and coboundaries, we refer the reader to the influential exposition~\cite{MacLane95} by MacLane.
For the corresponding topological setting, we refer to the outstanding paper~\cite{WaW015} by Wagemann and Wockel.

From this point forward, all topological spaces are assumed to be Hausdorff.
For a map $\pi: \B \to X$ between topological spaces, the fiber at $x \in X$ is denoted by $B_x := \pi^{-1}(x)$.

The identity element of a group is written as $e$. 

Let $G$ and $N$ be two groups.
Furthermore, let $p$ be a positive integer.
A map $f:G^p \to N$ is said to be \emph{normalized} if it equals $e_N$ whenever one of its arguments is $e_G$. 
The space of all normalized maps $G^p \to N$, referred to as the $p$-\emph{cochains}, is denoted by $C^p(G,N)$.

For a unital C\Star algebra $B$, the identity element is denoted by $1_B$, while $B$, $UB$ and $UZ(B)$
denote the group of invertible elements, the group of unitary elements of $B$, and group of unitary elements of the center $Z(B)$, respectively.

\subsection{Bundles}

\subsubsection{Banach bundles}

\begin{defi}[{\cite[Chap.~II, Sec.~13.4]{Fell1}}]
A \emph{Banach bundle} is a triple $(\B,X,\pi)$, where $\B$ is a topological space, $X$ is a locally compact space, and $\pi:\B \to X$ is a continuous open surjection, together with operations and norms making each \emph{fiber} $B_x$, for $x\in X$, into a Banach space, and satisfying the following conditions:
\begin{enumerate}[label=(B\arabic*)]
\item 
    $\B \ni s \mapsto \|s\| \in \mathbb{R}$ is continuous.   
\item 
    $\{(s,t)\in \B \times \B : \pi(s)=\pi(t)\} \ni (s,t) \mapsto s+t \in \B$ is continuous.
\item 
    $\B \ni s \mapsto \lambda s \in \B$ is continuous for every $\lambda \in \mathbb{C}$.
\item 
    If $\{s_i\}$ is any net of elements in $\B$ such that $\|s_i\|\to 0$ and $\pi(s_i)\to x$, then $s_i \to 0_x.$ 
\end{enumerate}
Notably, the topology of $\B$ restricted to $B_x$, for $x\in X$,  coincides with its norm topology.
Furthermore, the space of all continuous sections of $(\B,X,\pi)$ is denoted by $\Gamma(X,\B)$.
\end{defi}

\begin{example}
\label{ex:trivialBB}
Let $X$ be a locally compact space, let $B$ be a Banach space, let $\B := X \times B$, endowed with the product topology, and let $\pi : \B \to X$ be the canonical projection onto~$X$.
It is clear that $(\B,X,\pi)$ defines a Banach bundle, referred to as the \emph{trivial Banach bundle over $X$ with constant fiber~$B$}.
\end{example}

\begin{rem}
\label{rem:enough}
Let $(\B,X,\pi)$ be a Banach bundle.
By~\cite[Chap.~II, Rem.~13.19]{Fell1}, $(\B,X,\pi)$ has \emph{enough sections}, \ie, for each $s \in \B$ there exists $\sigma \in \Gamma(X,\B)$ such that $\sigma(\pi(s)) = s$.
In particular, $\{\sigma(x):\sigma \in \Gamma(X,\B)\} = B_x$ for every $x \in X$.
\end{rem}

The construction of specific Banach bundles typically arises in the following setting:

\begin{lemma}[Fell {\cite[Sec.~13.18]{Fell1}}]
\label{lem:bb_top}
Let $\B$ be a set and let $\pi:\B \to X$ be a surjection onto a locally compact space $X$ such that every $B_x$, for $x \in X$, is a Banach space.
Furthermore, let~$\Gamma$ be a complex linear space of sections of $(\B,\pi)$ such that
\begin{enumerate}[label=(\arabic*)]
\item 
    $\{f(x):f \in \Gamma\}$ is dense in $B_x$ for every $x \in X$, and
\item 
   $X \ni x \to \Vert f(x) \Vert \in \mathbb{R}$ is continuous from $X$ to $\mathbb{R}$ for every $f \in \Gamma$.
\end{enumerate}
Then there exists a unique topology on $\B$ that makes the triple $(\B,X,\pi)$ a Banach bundle such that $\Gamma \subseteq \Gamma(X,\B)$.
\end{lemma}

\begin{defi}
\label{def:equivalence_bb} 
Two Banach bundles $(\B^1,\pi^1,X)$ and $(\B^2,\pi^2,X)$ are said to be \emph{equivalent} if there exists a homeomorphism $\Psi: \B^1 \to \B^2$ such that
\begin{enumerate}[label=(\arabic*)]
\item 
    $\Psi(B^1_x) \subseteq B^2_x$ for every $x \in X$, and
\item 
    $\Psi_x := \Psi_{\vert B^1_x} : B^1_x \to B^2_x$ is a linear isometry for every $x \in X$.
\end{enumerate}
Such a map $\Psi$ is called an \emph{equivalence} between $(\B^1,\pi^1,X)$ and $(\B^2,\pi^2,X)$.
\end{defi}

\begin{lemma}[{\cite[Chap.~II, Sec.~13.16 \& 13.17]{Fell1}}]
\label{lem:bb_cont}
Let $(\B^1,\pi^1,X)$ and $(\B^2,\pi^2,X)$ be Banach bundles and let $\Psi: \B^1 \to \B^2$ be a map such that 
\begin{enumerate}[label=(\arabic*)]
\item 
    $\Psi(B^1_x) \subseteq B^2_x$ for every $x \in X$, and
\item 
    $\Psi_x := \Psi_{\vert B^1_x} : B^1_x \to B^2_x$ is a linear isometry for every $x \in X$.
\end{enumerate}
Furthermore, let~$\Gamma$ be a family of continuous sections of $\pi:\B^1 \to X$ such that
\begin{enumerate}[label=(\arabic*),resume]
\item 
    $\{\sigma(x):\sigma \in \Gamma\}$ has dense linear span in $B^1_x$ for every $x \in X$, and
\item 
    $\Psi \circ \sigma$ is a continuous sections of $\pi:\B^2 \to X$ for every $\sigma \in \Gamma$.
\end{enumerate}
Then $\Psi: \B^1 \to \B^2$ is an equivalence.
\end{lemma}

\subsubsection{Fell bundles}

\begin{defi}[{\cite[Chap.~VIII, Sec.~16.2]{Fell2}}]
\label{def:fb}
A \emph{Fell bundle} (\aka a C\Star algebraic bundle) is a Banach bundle $(\B,G,\pi)$, where $G$ is a locally compact group, equipped with a continuous associative multiplication $m: \B \times \B \to \B$, using the notation $st := m(s,t)$ for all $s,t \in \B$, and a continuous involution ${}^*:\B \to \B$ satisfying the following conditions:
\begin{enumerate}[label=(F\arabic*)]
\item
\label{cond:fb1}
    $B_g B_h \subseteq B_{gh}$ and the restriction $m_{g,h}:B_g \times B_h \to B_{gh}$ is bilinear for all $g,h\in G$.   
\item
\label{cond:fb2}
      $\|st\|\leq\|s\|\|t\|$ for all $s,t \in \B.$
\item
\label{cond:fb3}
    $B^*_g\subseteq B_{g^{-1}}$ and the restriction ${}^*:B_g\to B_{g^{-1}}$ is anti-linear for all $g\in G$.
\item
\label{cond:fb4}
    $(st)^*=t^*s^*$ for all $s,t \in \B$.    
\item
\label{cond:fb5} 
    $(s^*)^*=s$ for all $s \in \B$.
\item
\label{cond:fb6}
    $\|s^*s\|= \|s\|^2$ for all $s \in \B$.   
\item
\label{cond:fb7} 
    $s^*s\geq 0$ in $B_e$ for all $s \in \B$.   
\end{enumerate}
A Fell bundle $(\B,G,\pi)$ is called \emph{saturated} if $\spann\{B_g B_h\}$ is dense in $B_{gh}$ for all $g,h \in G$.
\end{defi}

\begin{rem}\label{rem:fb}
Let $(\B,G,\pi)$ be a Fell bundle.
\begin{enumerate}[label={\arabic*}.,ref=\ref{rem:fb}.{\arabic*}]
\item 
    If $\B$ is a discrete space and $G$ is a discrete group, the the topological requirements are vacuous.
    In this case, we occasionally refer to $(\B,G,\pi)$ as a \emph{discrete Fell bundle}.
\item 
    Conditions~\ref{cond:fb1}-\ref{cond:fb6} imply that $B_e$ is a C\Star algebra with respect to the restricted operations, known as the \emph{unit fiber}, and item~\ref{cond:fb7} refers to the corresponding standard order relation on $B_e$.
\item
\label{rem:fb:unital}
    If $B_e$ is unital and $(\B,G,\pi)$ is saturated, then $\spann\{B_g B_{g^{-1}}\} = B_e$ for all $g \in G$.
\item
\label{rem:fb:MSE}
    Conditions~\ref{cond:fb1} and~\ref{cond:fb3} imply that each $B_g$, $g \in G$, is almost a Morita equivalence $B_e$-bimodule with respect to the canonical left and right actions and the inner products ${}_{B_e}\langle s,t \rangle := st^*$ and $\langle s,t \rangle_{B_e} := s^*t$ for all $s,t \in B_g$.
    In fact, the only missing condition is that the linear span of the inner products need not be dense.
    Therefore, if $(\B,G,\pi)$ is saturated, then each $B_g$, $g \in G$, is a Morita equivalence $B_e$-bimodule.
\item 
    If $B_e$ is unital, then $1_{B_e}s = s = s1_{B_e}$ for all $s \in \B$.
\item
\label{rem:fb:innerprod}
    The map $\{(s,t)\in \B \times \B : \pi(s)=\pi(t)\} \ni (s,t) \mapsto {}_{B_e}\langle s,t \rangle \in B_e$ is continuous, as it is the restriction of the continuous map $m \circ (\id_\B \times {}^* \times) : \B \times \B \to \B$.
    Similarly, the map $\{(s,t)\in \B \times \B : \pi(s)=\pi(t)\} \ni (s,t) \mapsto \langle s,t \rangle_{B_e} \in B_e$ is continuous.
\item
\label{rem:fb:multi}
    Each map $m_{g,h}:B_g \times B_h \to B_{gh}$, for $g,h \in G$, in item~\ref{cond:fb1} factors through a $B$-bimodule homomorphism $m_{g,h}:B_g \otimes_{B_e} B_h \to B_{gh}$ (by a slight abuse of notation) that preserves both the left and right inner products. 
    As a consequence, if $(\B,G,\pi)$ is saturated, then $B_g \otimes_{B_e} B_h \cong B_{gh}$ as Morita equivalence $B$-bimodules for all $g,h$.
\item
\label{rem:fb:inv}
    Suppose that $(\B,G,\pi)$ is saturated.
    Let $g \in G$ and let $s \in B_g$.
    According to the previous item, there exist finitely many elements $s_i \in B_g$ and $t_i \in B_{g^{-1}}$ such that $1_{B_e} = \sum_{i} m_{g,g^{-1}}(s_i \otimes t_i$).
    Hence, $s^* = \sum_{i} \langle s,s_i \rangle_B t_i$.
\item
\label{rem:fb:multi_cont}
    Let $\Gamma$ be family of continuous sections of $\pi : \B \to G$ such that $\{\sigma(g):\sigma \in \Gamma\}$ is dense in $B_g$ for every $g \in G$.
    By~\cite[Chap.~VIII, Sec.~2.4]{Fell2}, the multiplication $m: \B \times \B \to \B$ is continuous if and only if the induced map $m \circ (\sigma,\sigma'): G \times G \to \B$ is continuous for every $(\sigma,\sigma') \in \Gamma \times \Gamma$.
\item
\label{rem:fb:inv_cont}
    Let $\Gamma$ be as in the previous item.
    \cite[Chap.~VIII, Sec.~3.2]{Fell2} shows that the involution ${}^*:\B \to \B$ is~continuous if and only if the induced map ${}^* \circ \sigma: G \to \B$ is continuous for every $\sigma \in \Gamma$.
\item
\label{rem:fb:sections}
    Let $\lambda$ be a left Haar measure on $G$ and let $\Delta$ be the corresponding modular function.
    The space $\Gamma_c(G,\B)$ of continuous, \emph{compactly supported sections} of $\pi : \B \to G$ forms a \Star algebra \wrt the multiplication and involution defined for $\sigma,\sigma' \in \Gamma_c(G,\B)$ respectively by
	\begin{align*}
		(\sigma \ast \sigma')(g):=\int_G \sigma(h) \sigma'(h^{-1}g) \, d\lambda(h)
		\quad
		\text{and}
		\quad
		\sigma^*(g):=\Delta (g^{-1}) \sigma(g^{-1})^*.
	\end{align*}
\end{enumerate}
\end{rem}

\begin{example}
\label{ex:trivialFB}
Let $G$ be a locally compact space, let $B$ be a C\Star algebra, and let $(\B,G,\pi)$ denote the trivial Banach bundle over $G$ with constant fiber $B$ (see Example~\ref{ex:trivialBB}).
Equipping $\B$ with the multiplication $m: \B \times \B \to \B$, defined by $m((g_1, b_1),(g_2, b_2)) := (g_1 g_2, b_1 b_2)$, and the involution ${}^: \B \to \B$, defined by $(g, b)^ := (g^{-1}, b^*)$, makes it a Fell bundle. 
Clearly, $(\B, G, \pi)$ is saturated. 
This Fell bundle is referred to as the \emph{trivial Fell bundle over $G$ with constant fiber $B$}.
\end{example}

\begin{example}
\label{ex:flb}
Let $G$ and $H$ be locally compact groups, let $\mathbb{T}$ denote the 1-torus, and let $p:H \to G$ be a locally trivial principal $\mathbb{T}$-bundle (see, \eg,~\cite{Huse94}).
Define $\B:= H \times_\mathbb{T} \mathbb{C} := (H \times \mathbb{C})/\mathbb{T}$ as the quotient space of $H \times \mathbb{C}$ under the left action of $\mathbb{T}$, defined for $t \in \mathbb{T}$, $h \in H$, and $z \in \mathbb{C}$, by $t\acts (h,z) := (h \acts z,zt)$.
This action is free and proper, and the $\mathbb{T}$-orbit of $(h,z)$ is denoted by $[h,z]$.
The map $\pi: \B \to G$, defined by $\pi[(h,z)] := p(h)$, is a line bundle over $H$.
To endow $\pi: \B \to G$ with the structure of a Fell bundle, one defines a multiplication and an involution as follows:
\begin{gather*}
    m: \B \times \B \to \B,
    \qquad
    m([h_1,z_1],[h_2,z_2]) := [h_1h_2,z_1z_2],
    \\
    {}^*: \B \to \B,
    \qquad
    [h,z]^* := [h^{-1},\bar{z}].
\end{gather*}
It is evident that $(\B,G,\pi)$ is saturated and that $B_g \cong \mathbb{C}$ for every $g \in G$.
Such Fell bundles are referred to as \emph{Fell line bundles}.
Notably, any Fell bundle with one-dimensional fibers is known to be of this type.
This example will be utilized multiple times.
\end{example}

\begin{example}
Let $G$ be a discrete group.
A C\Star algebra $A$ is said to be $G$-graded if it comes equipped with a linearly independent family $(A_g)_{g \in G}$ of closed subspaces of $A$ satisfying the following conditions:
\begin{enumerate}[label=(\arabic*)]
\item 
    $A_g A_h \subseteq A_{gh}$ for all $g,h\in G$.
\item
    $A^*_g \subseteq A_{g^{-1}}$ for all $g\in G$.
\item 
    The algebraic direct sum $\bigoplus_{g \in G}^{\text{alg}} A_g$ is dense in~$A$.
\end{enumerate}
Some observations are pertinent here:

\begin{enumerate}[label={\arabic*}.]
\item
    Every graded C\Star algebra gives rise to a discrete Fell bundle.
\item 
    Any strongly continuous action of a compact Abelian group on a unital C\Star algebra induces a grading by its dual group (see, \eg,~\cite[Thm.~4.22]{HoMo06}), thereby yielding a Fell bundle over the latter.
Moreover, this construction exhausts all Fell bundles over discrete Abelian groups (see, \eg,~\cite{Rae18}).
\item 
    Free actions of compact Abelian groups on unital C\Star algebras are in one-to-one correspondence with saturated Fell bundles over discrete Abelian groups (\cf~\cite{SchWa15}). 
\end{enumerate}
\end{example}

\begin{defi}
\label{def:equivalence_fb} 
Two saturated Fell bundles $(\B^1,\pi^1,G)$ and $(\B^2,\pi^2,G)$ with unit fiber $B$ are said to be \emph{equivalent} if there exists an equivalence $\Psi: \B^1 \to \B^2$ between the underlying Banach bundles (see Definition~\ref{def:equivalence_bb}) that is normalized and preserves their multiplicative and involutive structures, \ie,
\pagebreak[3]
\begin{enumerate}[label=(\arabic*)]
\item
	$\Psi_e = \id_B$,
\item 
    $\Psi_g(s) \Psi_h(t) = \Psi_{gh}(st)$ for all $g,h \in G$, $s \in B^1_g$, and $t \in B^1_h$, and
\item 
     $\Psi_g(s)^* = \Psi_{g^{-1}}(s^*)$ for all $g \in G$, $s \in B^1_g$.
\end{enumerate}
Such a map $\Psi$ is called an \emph{equivalence} between $(\B^1,\pi^1,G)$ and $(\B^2,\pi^2,G)$.
Notably, all of the maps  $\Psi_g$,  for $g \in G$, are necessarily isomorphisms of Morita equivalence $B$-bimodules.
\end{defi}

\subsubsection{Constructing Fell Bundles}\label{sec:construction_fb}

The process of constructing specific Fell bundles is typically carried out in the following way:
Let $\B$ be a set and let $\pi:\B \to G$ be a surjection onto a locally compact group $G$ such that each fiber $B_g$, for $g \in G$, is a Banach space.
Assume that $\B$ is equipped with an associative multiplication $m: \B \times \B \to \B$ and an involution ${}^*:\B \to \B$, satisfying all axioms of Definition~\ref{def:fb} (except the topological requirements).
In other words, when $G$ is considered as a discrete group, the triple $(\B,G,\pi)$ forms a discrete Fell bundle.
Furthermore, let~$\Gamma$ be a \emph{fundamental family of sections of $\pi:\B \to G$}, meaning a \Star algebra of compactly supported sections of $\pi:\B \to G$ (see Remark~\ref{rem:fb:sections}) such that
\begin{enumerate}[label=(\arabic*)]
\item 
    $\{\sigma(g):\sigma \in \Gamma\}$ is dense in $B_g$ for every $g \in G$, and
\item 
    $G \ni g \mapsto \Vert \sigma(g) \Vert \in \mathbb{R}$ is continuous for every $\sigma \in \Gamma$.
\end{enumerate}

\begin{lemma}[{\cite[Prop.~2.5]{Tak14}}]
\label{lem:fb_top}
Under the above hypotheses, there exists a unique topology on $\B$ that makes $\pi:\B \to G$ a Fell bundle, and all sections in $\Gamma$ continuous, \ie, $\Gamma \subseteq \Gamma_c(G,\B)$.
\end{lemma}
%

\begin{rem}
Since $\Gamma$ is, in particular, a complex linear space, there exists a unique topology on $\B$ that turns $\pi:\B \to G$ into a Banach bundle (see Lemma~\ref{lem:bb_top}). 
The condition that $\Gamma$ is a \Star algebra of compactly supported sections of $\pi:\B \to G$ is then employed to establish the continuity of the multiplication and involution with respect to this topology.
\end{rem}

\begin{rem}
\label{rem:fb_top}
Let $(\B,G,\pi)$ be a Fell bundle.
Then $\Gamma_c(G,\B)$ is a fundamental family of sections of $(\B,G,\pi)$ (see Remark~\ref{rem:fb:sections}).
Moreover, one can recover a base for the topology of $\B$ via $\Gamma_c(G,\B)$.
\end{rem}

\subsection{Automorphisms of Morita equivalence bimodules}\label{sec:autME}

Let $B$ be a unital C\Star algebra and let $M$ be a Morita equivalence $B$-bimodule.
An automorphism of $M$ is a $B$-bimodule isomorphism from $M$ to itself that preserves both the left and right inner products.
We let $\Aut(M)$ stand for the group of automorphisms of $M$.
In this brief section, we present two insightful statements on automorphisms of $M$, which, although likely familiar to experts, seem to lack explicit discussion in the literature.

\begin{lemma}[{see, \eg,~\cite[Prop.~5.4 \& Cor.~5.5]{SchWa15}}]
\label{lem:autME}
Let $B$ be a unital C\Star algebra and let $M$ be a Morita equivalence $B$-bimodule.
Then the following assertions hold:
\begin{enumerate}[label={(\roman*)},ref=\ref{lem:autME}.{(\roman*)}]
\item\label{autME}
    The map $\varphi_M: UZ(B) \to \Aut(M)$, defined by $\varphi_M(u)(m) := um$, is a group isomorphism.
\item\label{isoME}
    For $u \in UZ(B)$ consider the map $\Psi_u \in \Aut(M)$ defined by $\Psi_u(m) := mu$ for all $m \in M$.
    Then the map $\psi_M: UZ(B)\to UZ(B)$, defined by $\psi_M(u) := \varphi_M^{-1}(\Psi_u)$, is a group automorphism.
\end{enumerate}
\end{lemma}

\subsection{The Picard group}\label{sec:pic}

Let $B$ be a unital C\Star algebra. 
Morita equivalence $B$-bimodules $M_1$ and $M_2$ are called equivalent if there exists Morita equivalence $B$-bimodule isomorphism $\Psi:M_1 \to M_2$, \ie, a $B$-bimodule isomorphism that preserves both the left and right inner products.
The equivalence class of a Morita equivalence $B$-bimodule $M$ is written as $[M]$.

The set of equivalence classes of Morita equivalence $B$-bimodules forms an Abelian group with respect to the internal tensor product.
This group, denoted by $\Pic(B)$, is referred to as the \emph{Picard group} of $B$ (see, \eg,~\cite{BrGrRi77}).

The group $\Out(B) := \Aut(B)/\Inn(B)$, commonly called the outer automorphism group of~$B$, is a subgroup of $\Pic(B)$.
Indeed, for $\alpha \in \Aut(B)$ let $M_\alpha$ denote the vector space $B$ endowed with the canonical left Hilbert $B$-module structure and the right Hilbert $B$-module structure defined by $s \acts b := s \alpha(b)$ and $\langle s, t \rangle_B := \alpha^{-1}(s^*t)$ for all $s,t \in M_\alpha$ and $b \in B$.
It is straightforwardly checked that $M_\alpha$ is a Morita equivalence $B$-bimodule.
Moreover, for $\alpha,\beta \in \Aut(B)$ we have $M_\alpha \otimes_{B} M_\beta \cong M_{\alpha \circ \beta}$.
We thus get a group homomorphism from $\Aut(B)$ to $\Pic(B)$.
As $M_\alpha$ is equivalent to $B$ if and only if $\alpha$ is inner, it follows that $\Out(B)$ is a subgroup of $\Pic(B)$ as claimed.
Noteworthily, if $B$ is separable and stable, then the injection $\Out(B) \to \Pic(B)$ in fact turns out to be an isomorphism (see~\cite[Sec.~3]{BrGrRi77}).  


The following statement is central to our work and provides an invariant for saturated Fell bundles, expressed in terms of their underlying group and unit fiber:

\begin{lemma}
\label{lem:pic}
If $(\B,G,\pi)$ is a Fell bundle, then the map $\psi_\B:G \to \Pic(B_e)$, defined by $\psi_\B(g):=[B_g]$, is a well-defined group homomorphism.
\end{lemma}
\begin{proof}
By Remark~\ref{rem:fb:MSE}, $\psi_\B$ is well-defined.
To show that $\psi_\B$ is a group homomorphism, let $g,h \in G$.
Since $(\B,G,\pi)$ is saturated, $B_g \otimes_{B_e} B_h \cong B_{gh}$ as Morita equivalence $B$-bimodules (see Remark~\ref{rem:fb:multi}).
Therefore $\psi_\B(gh)=[B_{gh}]=[B_g \otimes_{B_e} B_h]=[B_g][B_h]=\psi_\B(g)\psi_\B(h)$.
\end{proof}

\begin{example}
\label{ex:pic_flb}
By~\cite[Sec.~3]{BrGrRi77}, $\Pic(\mathbb{C})$ is trivial. Therefore, if $(\B,G,\pi)$ is a Fell line bundle (see Example~\ref{ex:flb}), then the induced group homomorphism $\psi_\B: G \to \Pic(\mathbb{C})$ is also trivial.
\end{example}

The group homomorphism $\psi_\B:G \to \Pic(B_e)$ in Lemma~\ref{lem:pic}, henceforth referred to as the \emph{Picard homomorphism associated with $(\B,G,\pi)$}, does not determine the Fell bundle structure of $(\B,G,\pi)$ up to equivalence.
The main objective of Section~\ref{sec:construction} is to identify the additional data required to fully determine this structure.

\subsection{Topological group cohomology}\label{sec:topcoho}

Let $G$ be a topological group and let $A$ be an Abelian topological group $A$ equipped with a continuous $G$-action $G \times A \to A$ defined by a group homomorphism $S: G \to \Aut(A)$.
For a positive integer $p$ let $C^p_c(G,A)$ denote the set of continuous $p$-cochains $G^p \to A$ and let $d_S:C^p_c(G,A) \to C^{p+1}_c(G,A)$ be the ordinary group differential given by 
\begin{align*}
    (d_Sf)(g_0,\ldots,g_p)
    &:=
    S(g_0)(f(g_1,\ldots,g_p))
    \\
    &+\sum^{p}_{j=1}(-1)^jf(g_0,\ldots ,g_{j-1}g_j,\ldots,g_p)+(-1)^{p+1}f(g_0,\ldots,g_{p-1}).
\end{align*}
It is readily seen that $d_S(C^p_c(G,A))\subseteq C^{p+1}_c(G,A)$.
A sub-complex of the standard group cohomology complex $(C^{\bullet}_p(G,A),d_S)$ is thus obtained.
Since $d^2_S=0$, the space $Z^p_c(G,A)_S := \ker({d_S}_{\vert C^p_c(G,A)})$ of \emph{continuous $p$-cocycles} contains the space $B^p_c(G,A)_S := d_S(C^{p-1}_c(G,A))$ of \emph{continuous $p$-coboundaries}.
The quotient $H^p_c(G,A)_S := Z^p_c(G,A)_S / B^p_c(G,A)_S$ is the $p^\text{th}$ \emph{(global) continuous cohomology group of $G$ with values in the $G$-module $A$}.

\section{Constructing Fell bundles via factor systems}
\label{sec:construction}

Let $G$ be a locally compact group, let $B$ be a unital C\Star algebra, and let $\psi:G \to \Pic (B)$ be a group homomorphism. 
In this section, we propose a framework for demonstrating the existence of a saturated Fell bundle
$(\B,G,\pi)$ with unit fiber $B$ and Picard homomorphism $\psi$ (see~Lemma~\ref{lem:pic}).

For a start, we choose a representative $B_g$ of $\psi(g)$ for every $g \in G$, especially $B_e:=B$, and consider the disjoint union $\B := \cup_{g\in G}B_g$ together with the canonical projection $\pi:\B \to G$.

To establish $(\B,G,\pi)$ as a Fell bundle, we aim to apply Lemma~\ref{lem:fb_top}, which requires, in particular, the existence of an associative multiplication $m: \B \times \B \to \B$ and an involution ${}^*:\B \to \B$ satisfying Conditions~\ref{cond:fb1}-\ref{cond:fb7}.
Moving forward, we focus on constructing such maps, a purely algebraic task, and therefore treat $G$ as a discrete group.

As $\psi$ is a group homomorphism, we have Morita equivalence $B$-bimodule isomorphisms $\Psi_{g,h}: B_g \otimes_B B_h\to B_{gh}$ for all $g,h \in G$.
In general, $\psi$ does not impose any specific relations among the family of maps $(\Psi_{g,h})_{g,h \in G}$.
However, when derived from a saturated Fell bundle $(\B,G,\pi)$ with unit fiber $B$, we can set $\Psi_{g,h} := m_{g,h}$ for all $g,h \in G$ (see Remark~\ref{rem:fb:multi}).
In this case, the associativity of $m: \B \times \B \to \B$ implies that 
\begin{align}
\label{eq:assoc}
    \Psi_{gh,k} \circ (\Psi_{g,h} \otimes \id_k) = \Psi_{g,hk} \circ (\id_g \otimes \Psi_{h,k})
    &&
    \forall g,h,k \in G.
\end{align}
This motivates the introduction of the following notion:

\begin{defi}\label{def:fs}
Let $G$ be a discrete group, let $B$ be a unital C\Star algebra, and let $\psi:G \to \Pic (B)$ be a group homomorphism.
A \emph{factor system for $\psi$} is a family $(B_g,\Psi_{g,h})_{g,h \in G}$, where
\begin{itemize}
\item 
    $B_g$, for $g \in G$, is a representative of $\psi(g)$, with $B_e:=B$, and
\item 
    $\Psi_{g,h}: B_g \otimes_B B_h\to B_{gh}$, for $g,h \in G$, is a Morita equivalence $B$-bimodule isomorphism, with $\Psi_{e,e}:=\id_B$,
\end{itemize}
that satisfies the \emph{associativity condition} in~\eqref{eq:assoc}.
\end{defi}

\begin{rem}
\label{rem:fs}
Let $G$ be a discrete group and let $B$ be a unital C\Star algebra. 
\begin{enumerate}[label={\arabic*}.,ref=\ref{rem:fs}.{\arabic*}]
\item
\label{rem:fs:induced}
    As indicated prior to Definition~\ref{def:fs}, each saturated Fell bundle $(\B,G,\pi)$ allows for a factor system of the form $(B_g,m_{g,h})_{g,h \in G}$ for $\psi_\B$ (see~Lemma~\ref{lem:pic}).
\item 
    Let $\psi:G \to \Pic (B)$ be a group homomorphism.
    The existence of a factor system for $\psi$ imposes a non-trivial cohomological condition, which we will characterize in Section~\ref{sec:non-emptiness}, along with presenting cohomological methods for the construction of factor systems, even in the absence of a saturated Fell bundle.
\item 
    Let $\psi:G \to \Pic (B)$ be a group homomorphism and let $(B_g,\Psi_{g,h})_{g,h \in G}$ be a factor system for $\psi$.
    For each $g \in G$, the maps $\Psi_{e,g}$ and $\Psi_{g,e}$ correspond to the left and right actions of $B$ on $B_g$, respectively.
\end{enumerate}
\end{rem}

We now show that each factor system, as defined in Definition~\ref{def:fs}, gives rise to a discrete saturated Fell bundle with the specified unit fiber C\Star algebra and Picard homomorphism.
To this end, let $G$ be a discrete group, let $B$ be a unital C\Star algebra, let $\psi:G \to \Pic (B)$ be a group homomorphism, and let $(B_g,\Psi_{g,h})_{g,h \in G}$ be a factor system for $\psi$.
Define $\B := \cup_{g \in G} B_g$ and let $\pi:\B \to G$ denote the canonical projection.

As the first step, we endow $\B$ with the multiplication $m: \B \times \B \to \B$, defined for $s \in B_g$, for $g \in G$, and $t \in B_h$, for $h \in G$, by 
\begin{align}
\label{eq:multi}
    m(s,t) := \Psi_{g,h}(s \otimes t).
\end{align}
It is immediate that $m$ is associative and satisfies Conditions~\ref{cond:fb1} and~\ref{cond:fb2}.
Furthermore, $\spann\{B_g B_h\}$ is dense in $B_{gh}$ for all $g,h \in G$, because we are dealing with isomorphisms.

Next, we equip $(\B,G,\pi)$ with an involution.
To achieve this, let $g \in G$ and let $s \in B_g$, and consider 
the auxiliary map $\lambda_s: B_{g^{-1}} \to  B$, defined by $\lambda_s(t) := \Psi_{g,g^{-1}}(s \otimes t)$.
This map is adjointable, as it results from composing the left creation operator $\theta_s: B_{g^{-1}} \to B_g \otimes B_{g^{-1}}$, defined by $\theta_s(t)= s\otimes t$, with the map $\Psi_{g,g^{-1}}$.
Specifically, for $b \in \B$, the adjoint takes the explicit form $\lambda_s^*(b)= \sum_{i} \langle s,s_i \rangle_B t_i$, where $\Psi_{g,g^{-1}}^*(b) = \sum_{i} s_i \otimes t_i$.
With this setup, we define the involution on $(\B,G,\pi)$ as follows (\cf~Remark~\ref{rem:fb:inv}):
\begin{align}
\label{eq:inv}
   s^* := \lambda_s^*(1_B)
\end{align}
We see at once that the involution is well-defined and satisfies Condition~\ref{cond:fb3}.
Note that the involution, when restricted to the unit fiber, coincides with the involution on $B$, as is easy to check.
To verify Condition~\ref{cond:fb5}, we assume that $\Psi^*_{g^{-1},g}(1_B)=\sum_i s_i\otimes t_i$.
Then
\begingroup
\allowdisplaybreaks
\begin{align*}
    (s^*)^*&
    =\lambda^*_{s^*}(1_B)
    =\sum_i \langle s^*,s_i\rangle_B t_i
    =\sum_i \langle 1_B,\lambda_s(s_i)\rangle_B t_i
    =\sum_i \lambda_s(s_i) t_i 
    \\
    &=\sum_i\Psi_{g,g^{-1}}(s \otimes s_i) t_i 
    =\sum_i \Psi_{e,g}(\Psi_{g,g^{-1}}(s \otimes s_i)\otimes t_i) 
    \\
    &=\sum_i \Psi_{e,g} \circ (\Psi_{g,g^{-1}} \otimes \id_g) (s \otimes s_i\otimes t_i) 
    \\
    \overset{\eqref{eq:assoc}}&{=} \sum_i \Psi_{g,e} \circ (\id_g \otimes\Psi_{g^{-1},g}) (s \otimes s_i\otimes t_i) 
    \\
    &=\sum_i \Psi_{g,e}(s \otimes \Psi_{g^{-1},g} (s_i\otimes t_i))
    = \Psi_{g,e}(s \otimes 1_B)= s,
\end{align*}
\endgroup
which is the desired conclusion.
Furthermore, we have
\begin{align*}
    s^*s
    &= \langle 1_B,s^*s \rangle_B
    = \langle \lambda_{s^*}^*(1_B),s \rangle_B
    =  \langle (s^*)^*,s \rangle_B
    =  \langle s,s \rangle_B \geq 0,
\end{align*}
and thus $\Vert s^*s \Vert = \Vert \langle s,s \rangle_B \Vert = \Vert s \Vert^2$, with the last equality following from the fact that $B_g$ is a Morita equivalence $B$-bimodule.
This demonstrates that Conditions~\ref{cond:fb6} and~\ref{cond:fb7} hold.
We will prove Condition~\ref{cond:fb4} by means of the following technical lemma:

\begin{lemma}
\label{lem:cond:fb4}
Let $g,h\in G$, let $s\in B_g$, and let $t\in B_h$.
Then the map $\Omega_{st} := \lambda_{st} \circ \Psi_{h^{-1},g^{-1}}$ with domain $B_{h^{-1}} \otimes_B B_{g^{-1}}$ and range $B$ satisfies the following conditions:
\begin{enumerate}[label=(\roman*),ref=\ref{lem:cond:fb4}.{(\roman*)}]
\item\label{eq:omega1}  
    $\Omega_{st}(u \otimes v) = \lambda_s(\lambda_t(u)v)$ for all $u\in B_{h^{-1}}$ and $v\in B_{g^{-1}}$.
\item\label{eq:omega2} 
    $\Omega_{st}^*(b) = t^* \otimes \lambda^*_s(b)$ for all $b\in B$.
\end{enumerate}
\end{lemma}
\begin{proof}
\begin{enumerate}[label={(\roman*)}]
\item
    Let $u\in B_{h^{-1}}$ and let $v\in B_{g^{-1}}$.
    Then
    \begin{align*}
        \Omega_{st}(u \otimes v)
        &= 
        \Psi_{gh,h^{-1}g^{-1}}\circ (\Psi_{g,h}\otimes \id_{h^{-1}g^{-1}})(s \otimes t \otimes \Psi_{h^{-1},g^{-1}}(u \otimes v))
        \\
        \overset{\eqref{eq:assoc}}&{=} \Psi_{g,g^{-1}}(s \otimes \Psi_{e,g}(t \otimes \Psi_{h^{-1},g^{-1}}(u \otimes v)))
        \\
        \overset{\eqref{eq:assoc}}&{=} \Psi_{g,g^{-1}}(s \otimes \Psi_{e,g^{-1}}(\Psi_{h,h^{-1}}(t\otimes u) \otimes v))
        = 
        \Psi_{g,g^{-1}}(s \otimes \lambda_t(u)v) = \lambda_s(\lambda_t(u)v).
   \end{align*}
\item
    Let $b\in B$.
    Furthermore, let $u\in B_{h^{-1}}$ and let $v\in B_{g^{-1}}$.
    Then
,    \begin{align*}
        \langle \Omega_{st}(u\otimes v), b \rangle_B
        &=\langle \lambda_t(u)v, \lambda_s^*(b) \rangle_B
        \\
        &=\langle v, (\lambda_t(u))^*\lambda_s^*(b) \rangle_B
        \\
        &=\langle v, \langle \lambda_t(u),1_B \rangle_B \lambda_s^*(b)\rangle_B
        \\
        &=\langle v, \langle u,\lambda^*_t(1_B)\rangle_B \lambda_s^*(b) \rangle_B
        \\
        &=\langle v, \langle u,t^*\rangle_B \lambda_s^*(b) \rangle_B
        \\
        &=\langle u \otimes v, t^* \otimes \lambda_s^*(b) \rangle_B.
    \end{align*}
    It follows that $\Omega_{st}^*(b)=t^*\otimes \lambda^*_s(b)$, as claimed.
\end{enumerate}
\end{proof}

Let $g,h \in G$, let $s \in B_g$, and let $t \in B_h$.
By Lemma~\ref{eq:omega2}, we have $t^* \otimes s^*= \Omega_{st}^*(1_B)$.
Applying the map $\Psi_{h^{-1},g^{-1}}$ to both sides of this equation yields 
\begin{align*}
    t^*s^*=\lambda^*_{st}(1_B)=(st)^*,
\end{align*}
and this is precisely the assertion of Condition~\ref{cond:fb4}.
Summarizing the above, we have:

\begin{teo}
\label{teo:fs}
Let $G$ be a discrete group, let $B$ be a unital C\Star algebra, let $\psi:G \to \Pic (B)$ be a group homomorphism, and let $(B_g,\Psi_{g,h})_{g,h \in G}$ be a factor system for $\psi$.
Define $\B := \cup_{g \in G} B_g$ and let $\pi:\B \to G$ denote the canonical projection. 
Then, with the multiplication and involution defined by Equations~\eqref{eq:multi} and~\eqref{eq:inv}, respectively, $\pi: \B \to G$ acquires the structure of a (discrete) saturated Fell bundle, denoted by $\B((B_g,\Psi_{g,h})_{g,h \in G})$. 
\end{teo}

Using Theorem~\ref{teo:fs} in conjunction with Lemma~\ref{lem:fb_top}, we arrive at the desired result:

\begin{cor}
\label{cor:fs}
Under the assumptions of Theorem~\ref{teo:fs} with ``discrete'' replaced by ``locally compact'', suppose further that $\Gamma$ is a fundamental family of sections of $\pi:\B \to G$ (see Section~\ref{sec:construction_fb}), then there exists a unique topology on $\B$ that makes $\pi:\B \to G$ a Fell bundle, denoted by $\B((B_g,\Psi_{g,h})_{g,h \in G})_\Gamma$, and all sections in $\Gamma$ are continuous, \ie, $\Gamma \subseteq \Gamma(G,\B)$.
\end{cor}

\section{Classification of Fell bundles}
\label{sec:classification}

In the previous section, we demonstrated how factor systems lead to the construction of saturated Fell bundles (and vice versa). 
In this section, we establish a classification theory for this type of bundles using (topological) group cohomology.

\subsection{Classification of Fell bundles: The discrete case}
\label{sec:classification_alg}

For expediency we begin with the purely algebraic classification.
Here and subsequently, let $G$ be a discrete group and let $B$ be a unital C\Star algebra.

To begin with, we denote by $\Ext(G,B)$ the set of equivalence classes of saturated Fell bundles over $G$ with unit fiber~$B$ (see Definition~\ref{def:equivalence_fb}).
The equivalence class of a saturated Fell bundle $(\B,G,\pi)$ with unit fiber $B$ is denoted by $[(\B,G,\pi)]$.
It is evident that the map
\begin{align*}
    \Ext(G,B) \ni [(\B,G,\pi)] \mapsto \psi_\B \in \Hom(G,\Pic(B)),
\end{align*}
$\psi_\B$ being the Picard homomorphism associated with $(\B,G,\pi)$ (see Lemma~\ref{lem:pic}), is well-defined, yielding a partition of $\Ext(G,B)$ into the following subsets:

\begin{defi}
Under the above hypotheses, for $\psi \in \Hom(G,\Pic(B))$, we define:
\begin{align*}
    \Ext(G,B,\psi) := \{[(\B,G,\pi)]\in \Ext(G,B) : \psi_\B = \psi\}.
\end{align*}  
\end{defi}

Noteworthily, if $\psi:G \to \Pic(B)$ is the trivial group homomorphism, then $\Ext(G,B,\psi)$ is nonempty, as it contains the class of the trivial Fell bundle $G \times B$ (see Example~\ref{ex:trivialFB}).
However, for an arbitrary group homomorphism $\psi:G \to \Pic(B)$, the set $\Ext(G,B,\psi)$ may be empty. 
We defer addressing this issue until Section~\ref{sec:non-emptiness} and focus first on characterizing $\Ext(G,B,\psi)$ and its elements.

At the outset, we recall from~\cite[Prop.~5.6]{SchWa15} that the map $\mu: \Pic(B) \ni [M] \mapsto \psi_M \in \Aut(UZ(B))$ is a group homomorphism, where $\psi_M$ is the group automorphism described in Lemma~\ref{isoME}. Consequently, each $\psi \in \Hom(G,\Pic(B))$ induces a $G$-module structure on $UZ(B)$, defined by $\mu \circ \psi$, which facilitates the use of group cohomology, specifically
\begin{align*}
    H^p(G,UZ(B))_{\psi} := H^p(G,UZ(B))_{\mu \circ \psi}
\end{align*}
for all positive integers $p$ (for their definition, see, \eg,~\cite[Chap.~IV]{MacLane95}).
Having disposed of these preliminary steps, we now proceed as follows:

\begin{defi}~\label{def:equivalence} 
Let $G$ be a discrete group, let $B$ be a unital C\Star algebra, and let $\psi:G \to \Pic (B)$ be a group homomorphism.
Two factor systems
\begin{align*}
    (B^1_g,\Psi^1_{g,h})_{g,h \in G}
    &&
    \text{and}
    &&
    (B^2_g,\Psi^2_{g,h})_{g,h \in G}
\end{align*}
are said to be  \emph{equivalent} if there exists a family $(\Psi_g:B^1_g \to B^2_g)_{g \in G}$ of Morita equivalence $B$-bimodule isomorphisms satisfying the following compatibility condition:
\begin{align}\label{eq:comconfac}
	\Psi^2_{g,h}\circ(\Psi_g \otimes \Psi_h) = \Psi_{gh} \circ \Psi^1_{g,h}
    &&
    \forall g,h \in G.
\end{align}
Note that, necessarily, $\Psi_e = \id_B$.
\end{defi}

The following statement is readily checked, so we leave the details to the reader.

\begin{lemma}\label{lem:equivalence}
Let $G$ be a discrete group, let $B$ be a unital C\Star algebra, let $\psi:G \to \Pic (B)$ be a group homomorphism, and let $(B^1_g,\Psi^1_{g,h})_{g,h \in G}$, $(B^2_g,\Psi^2_{g,h})_{g,h \in G}$ be factor systems for $\psi$.
Then the following statements are equivalent:
\begin{enumerate}[label=(\alph*),ref=\ref{teo:equivalence}.{(\alph*)}]
\item\label{en:factor_sys_equiv}
	The factor systems are equivalent.
\item\label{en:dyn_sys_equiv}
	The Fell bundles $\B((B^1_g,\Psi^1_{g,h})_{g,h \in G})$ and $\B((B^2_g,\Psi^2_{g,h})_{g,h \in G})$ are equivalent.
\end{enumerate}
\end{lemma}

\begin{lemma}\label{lem:cocycles}
Let $G$ be a discrete group and let $B$ be a unital C\Star algebra. 
Suppose that $\psi:G \to \Pic(B)$ is a group homomorphism such that $\Ext(G,B,\psi) \neq \emptyset$. 
For each $g \in G$ choose a representative $B_g$ of $\psi(g)$, with $B_e := B$. 
Then the following assertions hold:
\begin{enumerate}[label={(\roman*)},ref=\ref{lem:autME}.{(\roman*)}]
\item
\label{lem:autME_rep}
    Each class in $\Ext(G,B,\psi)$ can be represented by a Fell bundle of the following form: $\B((B_g,\Psi_{g,h})_{g,h \in G})$ (see Theorem~\ref{teo:fs}).
\item
\label{lem:autME_Z}
    Any other Fell bundle $\B((B_g,\Psi'_{g,h})_{g,h \in G})$ representing a class in $\Ext(G,B,\psi)$ satisfies $\Psi'_{g,h} = \omega(g,h) \Psi_{g,h} := \varphi_{B_{gh}}(\omega(g,h)) \Psi_{g,h}$ (see Lemma~\ref{autME}), for $g,h \in G$, where $\omega \in Z^2(G,UZ(B))_{\psi}$.
\item
\label{lem:autME_B}
    The Fell bundles $\B((B_g,\Psi_{g,h})_{g,h \in G})$ and $\B((B_g,\omega(g,h)\Psi_{g,h})_{g,h \in G})$ are equivalent if and only if $\omega \in B^2(G,UZ(B))_{\psi}$.
\end{enumerate}
\end{lemma}
\begin{proof}
\begin{enumerate}[label={(\roman*)}]
\item
    Let $(\B',\pi',G)$ be a saturated Fell bundle representing a class in $\Ext(G,B,\psi)$.
    By Remark~\ref{rem:fs:induced}, it admits a factor system for $\psi$ of the form $(B'_g,m'_{g,h})_{g,h \in G}$.
    As $B'_g$ belongs to $\psi(g)$ for each $g \in G$, there exists a family $(\Psi_g:B_g \to B'_g)_{g \in G}$ of Morita equivalence $B$-bimodule isomorphisms.
    This yields a factor system $(B_g,\Psi_{g,h})_{g,h \in G}$ for $\psi$, equivalent to $(B'_g,m'_{g,h})_{g,h \in G}$, by putting $\Psi_{g,h} := \Psi_{gh}^* \circ m'_{g,h} \circ (\Psi_g \otimes_B \Psi_h)$, which is easily verified.
    Hence, the assertion follows from~Lemma~\ref{lem:equivalence}.
\item 
    Let $\B((B_g,\Psi'_{g,h})_{g,h \in G})$ be any other Fell bundle representing a class in $\Ext(G,B,\psi)$. 
    By Lemma~\ref{autME}, there exists a unique element $\omega(g,h) \in UZ(B)$ such that $\Psi'_{g,h} = \omega(g,h) \Psi_{g,h}$ for all $g,h$.
    It is readily verified that the corresponding map $\omega: G \times G \to UZ(B)$ takes the value $1_B$ whenever one of its arguments is $e$. 
    Furthermore, Lemma~\ref{isoME} implies that 
    \begin{align*}
        \Psi'_{g,hk} \circ (\id_g \otimes_B \Psi'_{h,k}) &= \Psi'_{g,hk} \circ (\id_g \otimes_B \omega(h,k) \Psi_{h,k})
        \\
        &= \Psi'_{g,hk} \circ (\id_g \omega(h,k) \otimes_B \Psi_{h,k})
        \\
        &= \Psi'_{g,hk} \circ (\psi_{B_g}(\omega(h,k)) \id_g \otimes_B \Psi_{h,k})
        \\
        &= \psi_{B_g}(\omega(h,k)) \Psi'_{g, hk} \circ (\id_g \otimes_B \Psi_{h,k})
        \\
        &= \psi_{B_g}(\omega(h,k)) \omega(g,hk) \Psi_{g, hk} \circ (\id_g \otimes_B \Psi_{h,k})
    \end{align*}
    for all $g,h,k\in G$.
    On the other hand, it is evident that
    \begin{align*}
        \Psi'_{gh,k} \circ (\Psi'_{g,h} \otimes \id_k) = \omega(g,h) \omega(gh,k) \Psi_{gh,k} \circ (\Psi_{g,h} \otimes \id_k)
    \end{align*}
    for all $g,h,k\in G$.
    Hence, $\omega \in Z^2(G,UZ(B))_{\psi}$ as claimed.
\item
    First, suppose that $\omega \in B^2(G,UZ(B))_{\psi}$, meaning there exists $\varpi \in C^1(G,UZ(B))$ such that $\omega = d_{\psi}\varpi$.
    Utilizing this element $\varpi$, a family $(\varphi_{B_g}(\varpi(g)):B_g \to B_g)_{g \in G}$ of Morita equivalence $B$-bimodules is obtained (see Lemma~\ref{autME}), through which an equivalence between $\B((B_g,\Psi_{g,h})_{g,h \in G})$ and $\B((B_g,\omega(g,h)\Psi_{g,h})_{g,h \in G})$ is implemented, as can be easily verified.
    This establishes the ``if'' direction.
    
    Conversely, suppose that $\B((B_g,\Psi_{g,h})_{g,h \in G})$ and $\B((B_g,\omega(g,h)\Psi_{g,h})_{g,h \in G})$ are equivalent.
    Lemma~\ref{autME} guarantees the existence of $\varpi \in C^1(G,UZ(B))$ that implements this equivalence.
    Moreover, from Equation~\eqref{eq:comconfac} it follows that $\omega=d_{\psi}\varpi$, \ie, $\omega \in B^2(G,UZ(B))_{\psi}$. 
    This proves the ``only if'' direction.
\end{enumerate}
\end{proof}

\begin{cor}\label{cor:class}
Let $G$ be a discrete group and let $B$ be a unital C\Star algebra. 
Suppose that $\psi:G \to \Pic(B)$ is a group homomorphism such that $\Ext(G,B,\psi)\neq\emptyset$. 
Then the map
\begin{gather*}
	H^2(G,UZ(B))_{\psi} \times \Ext(G,B,\psi) \to \Ext(G,B,\psi),
    \\
	\left([\omega],[\B((B_g,\Psi_{g,h})_{g,h \in G})]\right) \mapsto 
    [\B((B_g,\omega \Psi_{g,h})_{g,h \in G})]
\end{gather*}
is a well-defined simply transitive action.
\end{cor}

\begin{example}
\label{ex:cfld}
Let $G$ be a discrete group and let $B = \mathbb{C}$. 
By~\cite[Sec.~3]{BrGrRi77}, $\Pic(\mathbb{C})$ is trivial, which implies that there exists only the trivial group homomorphism $\psi: G \to \Pic(\mathbb{C})$. 
Fell bundles associated with $\psi$ certainly exist, are characterized by the property that every fiber is one-dimensional (\cf Example~\ref{ex:flb}).
For instance, the trivial Fell bundle over $G$ with constant fiber $\mathbb{C}$ provides an example (see Example~\ref{ex:trivialFB}).
Thus, Corollary~\ref{cor:class} shows that equivalence classes of Fell line bundles over $G$ are parametrized by $H^2(G, \mathbb{T})$.

For example, when $G = \mathbb{Z}$, $H^2(\mathbb{Z}, \mathbb{T})$ is trivial (see, \eg,~\cite[Chap.~VI]{MacLane95}).
Consequently, up to equivalence, there exists only one Fell line bundle over $\mathbb{Z}$, namely the trivial one.
In contrast, for $G = \mathbb{Z}^2$, $H^2(\mathbb{Z}^2, \mathbb{T})$ is isomorphic to $\mathbb{T}$ (see, \eg,~\cite[Prop.~II.4]{Neeb07b}).
\end{example}

\subsection*{Non-emptiness of $\Ext(G,B,\psi)$}
\label{sec:non-emptiness}

Let $\psi:G \to \Pic(B)$ be a group homomorphism.
In this section, we establish a cohomological criterion ensuring that $\Ext(G,B,\psi)$ is nonempty.
As discussed at the beginning of Section~\ref{sec:construction}, $\psi$ allows for the choice of a family $(B_g,\Psi_{g,h})_{g,h \in G}$, where
\begin{itemize}
\item 
    $B_g$, for $g \in G$, is a representative of $\psi(g)$, with $B_e:=B$, and
\item 
    $\Psi_{g,h}: B_g \otimes_B B_h\to B_{gh}$, for $g,h \in G$, is a Morita equivalence $B$-bimodule isomorphism, with $\Psi_{e,e}:=\id_B$, and such that $\Psi_{e,g}$ and $\Psi_{g,e}$, for $g \in G$, correspond to the left and right actions of $B$ on $B_g$, respectively. 
\end{itemize}
Notably, this family does not need to form a factor system for $\psi$ in the sense of Definition~\ref{def:fs}, \ie, the associativity condition~\eqref{eq:assoc} need not be satisfied.
However, we may instead consider the automorphisms
\begin{align*}
    d_\B \Psi_{g,h,k} := \Psi_{gh,k} \circ (\Psi_{g,h} \otimes \id_k) \circ (\id_g \otimes \Psi^*_{h,k}) \circ \Psi^*_{g,hk} : B_{ghk} \to B_{ghk}
\end{align*}
for all $g,h,k \in G$.
The family $(d_\B \Psi_{g,h,k})_{g,h,k \in G}$ can be interpreted as an obstruction to the associativity of the multiplication defined in~\eqref{eq:multi}.
On the other hand, Lemma~\ref{autME} implies that the family $(d_\B \Psi_{g,h,k})_{g,h,k \in G}$ can also be viewed as a normalized $UZ(B)$-valued 3-cochain on $G$, \ie, an element of $C^3(G,UZ(B))$, say $d_\B \Psi$. 
In fact, even more is true: $d_\B \Psi \in Z^3(G,UZ(B))_{\psi}$.
For brevity, we omit the detailed calculation here and instead refer to~\cite[Lem.~1.10(5)]{Neeb07}, for example.

\begin{lemma}
    The element $[d_\B \Psi] \in H^3(G,UZ(B))_{\psi}$ does not depend on any of the choices made.
\end{lemma}
\begin{proof}
We first show that $[d_\B \Psi]$ does not depend on the choice of the family $(\Psi_{g,h})_{g,h \in G}$.
To this end, let $(\Psi'_{g,h})_{g,h \in G}$ be another such choice.
By Lemma~\ref{autME}, there exists a unique element $\varpi(g,h) \in UZ(B)$ such that $\Psi'_{g,h} = \omega(g,h) \Psi_{g,h} := \varphi_{B_g}(\varpi(g,h))$ for all $g,h$.
Then
\begin{align*}
    \Psi'_{gh,k} \circ (\Psi'_{g,h} \otimes_B \id_k)
    =
    \varpi(gh,k) \varpi (g,h) d_\B \Psi_{g,h,k} (\Psi_{g,hk} \circ (\id_g \otimes_B \Psi_{h,k}))
\end{align*}
for all $g,h,k \in G$.
Moreover,
\begin{align*}
    d_\B \Psi'_{g,h,k}(\Psi'_{g,hk} \circ (\id_g \otimes_B \Psi'_{h,k}))
    =
    d_\B \Psi'_{g,h,k} \varpi(g,hk) \psi_{B_g}(\varpi(h,k)(\Psi_{g,hk} \circ (\id_g \otimes_B \Psi_{h,k}))
\end{align*}
for all $g,h,k \in G$.
Comparing these two expressions clearly establishes that the 3-cocycles $d_\B \Psi'$ and $d_\B \Psi$ are cohomologous.
We now demonstrate that $[d_\B \Psi]$ does not depend on the choice of the family~$(B_g)_{g \in G}$. 
For this purpose, let $(B'_g)_{g \in G}$ be another such choice, and let $(\Psi_g: B_g \to B'_g)_{g \in G}$ be a family of Morita equivalence $B$-bimodule isomorphisms. 
These induce a new family $(\Psi''_{g,h}: B'_g \otimes_B B'_h \to B'{gh})_{g,h \in G}$ of Morita equivalence $B$-bimodule isomorphisms, defined by $\Psi''_{g,h} := \Psi_{gh}^* \circ \Psi_{g,h} \circ (\Psi_g \otimes_B \Psi_h)$, for $g,h \in G$. 
A straightforward computation, analogously to the previous one, shows that $d_{\B'} \Psi'_{g,h,k} = d_{\B} \Psi_{g,h,k}$.
\end{proof}

\begin{defi}
Let $G$ be a discrete group, let $B$ be a unital C\Star algebra, and let $\psi:G \to \Pic(B)$ be a group homomorphism.
The element 
\begin{align*}
    \chi(\psi) := [d_\B \Psi] \in H^3(G,UZ(B))_{\psi}
\end{align*}
is called the \emph{characteristic class} of $\psi$.
\end{defi}

\begin{teo} 
\label{teo:H^3}
Let $G$ be a discrete group, let $B$ be a unital C\Star algebra, and let $\psi:G \to \Pic(B)$ be a group homomorphism.
$\Ext(G,B,\psi)$ is nonempty if and only if $\chi(\psi)$ vanishes.
\end{teo} 
\begin{proof}
First, suppose that $\Ext(G,B,\psi)$ is nonempty, and let $(\B,G,\pi)$ be a Fell bundle representing a class in $\Ext(G,B,\psi)$. 
By Remark~\ref{rem:fs:induced}, $(\B,G,\pi)$ admits a factor system for $\psi$ of the form $(B_g,m_{g,h})_{g,h \in G}$.
Since the multiplication is associative, it follows that
$\chi(\psi)$ vanishes, proving the ``only if'' direction.

Conversely, suppose that $\chi(\psi)$, represented by a family $(B_g,\Psi_{g,h})_{g,h \in G}$ of the above form, vanishes.
Then there exists an element $\varpi \in C^2(G,UZ(B))$ such that $d_\B\Psi = d_{\psi} \varpi^{-1}$, which yields a family $(\Psi'_{g,h} := \varphi_{B_g}(\varpi(g,h)) \Psi_{g,h}: B_g \otimes_B B_h \to B_{gh})_{g,h \in G}$ of Morita equivalence $B$-bimodule isomorphisms satisfying $d_\B\Psi'_{g,h,k} = \id_{ghk}$ for all $g,h,k \in G$, as is straightforwardly checked.
Consequently, $(B_g,\Psi'_{g,h})_{g,h \in G}$ is a factor system for $\psi$, and hence $\Ext(G,B,\psi)$ is nonempty by Theorem~\ref{teo:fs}.
This establishes the ``if'' direction, completing the proof.
\end{proof} 

We conclude this section with the introduction of the following convenient terminology:

\begin{defi}
\label{def:pic}
Let $G$ be a discrete group and let $B$ be a unital C\Star algebra.
A group homomorphism $\psi:G \to \Pic(B)$ is called a \emph{Picard homomorphism} if $\chi(\psi)$ vanishes, \ie, it admits a Fell bundle over $G$ with unit fiber $B$ and Picard homomorphism $\psi$.
\end{defi}

\subsection{Classification of Fell bundles: The topological case}
\label{sec:classification_top}

In this section, we address the problem of classifying saturated Fell bundles over locally compact groups.
Guided by Corollary~\ref{cor:class}, our strategy is to fix a topology and classify Fell bundles whose underlying Banach bundle structures are equivalent with respect to it.

In what follows, let $G$ be a locally compact group, unless explicitly stated otherwise, and let $B$ be a unital C\Star algebra.
However, we will sometimes consider $G$ as a discrete group to apply results from Section~\ref{sec:construction}. 
It should be clear from the context when $G$ is considered as discrete, or when it does not affect the discussion.
We commence as with the following observation:

\begin{lemma}
\label{lem:cont_mod_str}
Let $(\B,G,\pi)$ be a Fell bundle.
The induced map $G \times UZ(B_e) \to UZ(B_e)$, defined by $(g,z) \mapsto \mu(\psi_\B(g))(z)$, is continuous (see the discussion preceding Definition~\ref{def:equivalence}).
Furthermore, this map depends only on the equivalence class of $(\B,G,\pi)$.
\end{lemma}
\begin{proof}
    Let $g_0 \in G$.
    We show that the map in the lemma is continuous on a set of the form $U \times UZ(B_e)$ for some open neighbourhood $U$ of $g_0$.
    By~\cite[Lem.~A.1]{SchWa17}, there exist elements $x_1,\ldots,x_n \in B_{g}$ such that
    \begin{align*}
        \sum^n_{k=1} {}_{B_e}\langle x_k,x_k \rangle = 1_{B_e}.
    \end{align*}
    Since $(\B,G,\pi)$ has enough sections (see Remark~\ref{rem:enough}), there exist continuous sections $\sigma_k$ of $\pi:\B \to G$, for each $1\leq k \leq n$, such that $\sigma_k(g_0)=x_k$.
    Using these, the maps
    \begin{gather*}
        f: G \to B_e,
        \qquad
        f(g) := \sum^n_{k=1} {}_{B_e}\langle m(\sigma_k(g)),\sigma_k(g) \rangle
        \shortintertext{and}
        F: G \times B_e \to B_e,
        \qquad
        F(g,b) := \sum^n_{k=1} {}_{B_e}\langle m(\sigma_k(g),b),\sigma_k(g) \rangle
    \end{gather*}
    are defined.
    Their continuity follows from Remark~\ref{rem:fb:innerprod} and Remark~\ref{rem:fb:multi_cont}.
    Furthermore, for $g \in G$ and $z \in UZ(B_e)$, it is easily verified that $F(g,z) := \mu(\psi_\B(g))(z) \cdot f(g)$.
    With $B_e^\times$ being open and $f(g_0) =  1_{B_e}$, one can find an open neighbourhood $U$ around $g_0$ such that $f(U) \subseteq B_e^\times$.
    In consequence, on $U \times UZ(B_e)$, the equality $\mu(\psi_\B(g))(z) = F(g,z) \cdot (f(g))^{-1}$ holds, and hence the claim follows, as the right-hand side is a product of continuous maps.
    This completes the proof of the lemma.
\end{proof}

Noteworthy is the fact that Lemma~\ref{lem:cont_mod_str} enables the application of topological group~cohomology, specifically 
\begin{align*}
H^p_c(G,UZ(B))_{\psi} := H^p_c(G,UZ(B))_{\mu \circ \psi}
\end{align*}
for all positive integers $p$ (see Section~\ref{sec:topcoho}).

Next, we present the key concept of this section, which serves as the topological counterpart to Definition~\ref{def:pic}:

\begin{defi}
\label{def:pic_cont}
Let $G$ be a locally compact group and let $B$ be a unital C\Star algebra.
A Picard homomorphism $\psi:G \to \Pic(B)$ is called \emph{topological} if it admits a Fell bundle over $G$ with unit fiber $B$ and Picard homomorphism $\psi$.
\end{defi}

Clearly, a Picard homomorphism $\psi:G \to \Pic(B)$ is topological if and only if there exists a representative $(\B,G,\pi)$ in $\Ext(G,B,\psi)$ that admits a fundamental family of sections of $\pi:\B \to G$ (see Section~\ref{sec:construction_fb}).
It should be emphasized, however, that there may exist multiple nonequivalent fundamental families of sections of $\pi:\B \to G$, in the sense that the topologies they induce on $\B$ may not be equivalent.
The equivalence class of a fundamental family $\Gamma$ of sections of a Fell bundle is denoted by $[\Gamma]$.

\begin{example}
\label{ex:ce}
Consider the case where $B = \mathbb{C}$. 
Since $\Pic(\mathbb{C})$ is trivial (see ~\cite[Sec.~3]{BrGrRi77}), there exists only the trivial group homomorphism $\psi: G \to \Pic(\mathbb{C})$ for any locally compact group $G$. 
Fell bundles associated with $\psi$ are characterized by the property that every fiber is one-dimensional and are referred to as Fell line bundles~(see Example~\ref{ex:flb}).

In general, Fell line bundles are locally trivial (see, \eg,~\cite[Prop.~2.3]{Du74}). 
However, the trivial Fell line bundle also serves as an example (see Example~\ref{ex:trivialFB}). 
This demonstrates that Fell bundles with the same Picard homomorphism may nevertheless have inequivalent topologies.
\end{example}

Following this, let $\psi:G \to \Pic(B)$ be a topological Picard homomorphism, let $(\B,G,\pi)$ be a saturated Fell bundle with unit fiber $B$ and Picard homomorphism $\psi$, and let $\Gamma$ be a fundamental family of sections of $\pi:\B \to G$.

\begin{defi}
\label{def:ext_cont}
Under the above hypotheses, we define $\Ext(G,B,\psi,[\Gamma])$ as the set of equivalence classes of saturated Fell bundles over $G$ with unit fiber~$B$, Picard homomorphism $\psi$, and whose underlying Banach bundles are equivalent with respect to the topology determined by $\Gamma$ (see Definition~\ref{def:equivalence_bb} and Definition~\ref{def:equivalence_fb}).

The equivalence class of such a saturated Fell bundle is simply denoted by enclosing it in square brackets, thereby suppressing $[\Gamma]$ for the sake of notational simplicity.
\end{defi}

We proceed as follows:
Let $\psi:G \to \Pic(B)$ be a topological Picard homomorphism, and let $(\B,G,\pi)$ be a saturated Fell bundle with unit fiber $B$ and Picard homomorphism $\psi$.
Recall that the topology of $\B$ is uniquely determined by $\Gamma := \Gamma_c(G,\B)$ (see Remark~\ref{rem:fb_top}), and that $(\B,G,\pi)$ allows for a factor system for $\psi$ of the form $(B_g,m_{g,h})_{g,h \in G}$ (see Remark~\ref{rem:fs:induced}).
By Lemma~\ref{lem:autME_Z}, any other class in $\Ext(G,B,\psi)$ can be represented by a Fell bundle of the form $\B((B_g,(m_\omega)_{g,h})_{g,h \in G})$, where $(m_\omega)_{g,h} := \omega(g,h) m_{g,h}$ for some $\omega \in Z^2(G,UZ(B))_{\psi}$.
If $\omega \in Z^2_c(G,UZ(B))_{\psi}$, then is straightforwardly verified that $\Gamma_c(G,\B)$ endowed with the multiplication and involution defined for $\sigma,\sigma' \in \Gamma_c(G,\B)$ by
\begin{gather*}
	(\sigma \ast \sigma')(g):=\int_G \omega(h,h^{-1}g) \sigma(h) \sigma'(h^{-1}g) \, d\lambda(h)
	\\
    \shortintertext{and}
	\sigma^*(g):=\Delta (g^{-1}) \psi_{B_g}(\omega(g^{-1},g)) \sigma(g^{-1})^*
\end{gather*}
forms a fundamental family of sections of $\B((B_g,(m_\omega)_{g,h})_{g,h \in G})$, thereby endowing it with the structure of a Fell bundle (see Corollary~\ref{cor:fs}).
Here, $\lambda$ denotes a left Haar measure on $G$, and $\Delta$ denotes the corresponding modular function.
Notably, the Banach bundles underlying $(\B,G,\pi)$ and $\B((B_g,(m_\omega)_{g,h})_{g,h \in G})$ are equivalent, as follows easily from Lemma~\ref{lem:bb_cont}.
Conversely, consider the multiplication $m_\omega: \B \times \B \to \B$, defined for $s \in B_g$, for $g \in G$, and $t \in B_h$, for $h \in G$, by 
\begin{align}
\label{eq:m_omega}
    m_\omega(s,t) := (m_\omega)_{g,h}(s \otimes t) = \omega(g,h) m_{g,h}(s \otimes t).
\end{align}

\begin{lemma}\label{lem:cocycles_mult}
Under the above hypotheses, if $m_\omega: \B \times \B \to \B$ is continuous (see Equation~\eqref{eq:m_omega}), then $\omega$ is continuous as well, \ie, $\omega \in Z^2_c(G,UZ(B))_{\psi}$.
\end{lemma}
\begin{proof}
    Let $g_0,h_0 \in G$.
    It suffices to prove that $\omega$ is continuous on some open neighbourhood of $(g_0,h_0)$.
    Our proof strategy is similar to that of Lemma~\ref{lem:cont_mod_str}:
    By~\cite[Lem.~A.1]{SchWa17}, there exist elements $x_1,\ldots,x_n \in B_{g_0h_0}$ such that
    \begin{align*}
        \sum^n_{k=1} {}_{B_e}\langle x_k,x_k \rangle = 1_{B_e}.
    \end{align*}
    Since $(\B,G,\pi)$ is saturated, for each $1\leq k \leq n$, there exist a finite index set $I_k$, along with elements $\{s^k_i\}_{i \in I_k} \subseteq B_{g_0}$ and $\{t^k_i\}_{i \in I_k} \subseteq B_{h_0}$, such that $\sum_{i \in I_k} m(s^k_i, t^k_i)$ approximates $x_k$ arbitrarily well.
    The openness of $B_e^\times$ guarantees that it is possible to choose the elements $\{s^k_i,t^k_i\}$
    such that the sum
    \begin{align*}
        \sum^n_{k=1} \sum_{i,j \in I_k} {}_{B_e}\langle m(s^k_i, t^k_i),m(s^k_j, t^k_j) \rangle
    \end{align*}
    is invertible.
    Given that $(\B,G,\pi)$ has enough sections (see Remark~\ref{rem:enough}), there exist continuous sections $\sigma^k_i$ and $\tau^k_i$ of $\pi:\B \to G$, for all $1\leq k \leq n$ and $i \in I_k$, such that $\sigma^k_i(g_0)=s^k_i$ and $\tau^k_i(h_0)=t^k_i$.
    Now, consider the map
    \begin{align*}
        F: G \times G \to B_e,
        &&
        F(g,h) := \sum^n_{k=1} \sum_{i,j \in I_k} {}_{B_e}\langle m(\sigma^k_i(g), \tau^k_i(h)),m(\sigma^k_j(g), \tau^k_j(h)) \rangle.
    \end{align*}
    It follows from Remark~\ref{rem:fb:innerprod} and Remark~\ref{rem:fb:multi_cont} that $F$ is continuous.
    Consequently, there exist open neighbourhoods $U$ of $g_0$ and $V$ of $h_0$ such that $F(U \times V) \subseteq B_e^\times$.
    Next, consider the map    
    \begin{align*}
        F_\omega: G \times G \to B_e,
        &&
        F_\omega(g,h) := \sum^n_{k=1} \sum_{i,j \in I_k} {}_{B_e}\langle m_\omega(\sigma^k_i(g), \tau^k_i(h)),m(\sigma^k_j(g), \tau^k_j(h)) \rangle.
    \end{align*}
    Using the continuity of $m_\omega$, a similar reasoning as for $F$ above shows that $F_\omega$ is continuous.
    Furthermore, for all $g,h \in G$, the relation $F_\omega(g,h) = \omega(g,h) F(g,h)$ is satisfied.
    On $U \times V$, where $F(g,h) \in B_e^\times$, this implies that
    \begin{align*}
        \omega_{\vert U \times V} = {F_\omega}_{\vert U \times V} \cdot (F_{\vert U \times V})^{-1}.
    \end{align*}
    Thus, $\omega$ is continuous on $U \times V$ as it is the product of continuous maps.
    This is the desired conclusion, and so the proof is complete.
\end{proof}

With the groundwork laid, we are now in a position to state the topological versions of Lemma~\ref{lem:cocycles} and Corollary~\ref{cor:class}. The proofs, relying on this groundwork, are straightforward and left to the reader.

\begin{lemma}
\label{lem:cocycles_top}
Let $G$ be a locally compact group, let $B$ be a unital C\Star algebra, and let $\psi:G \to \Pic(B)$ be a topological Picard homomorphism.
Moreover, let $(\B,G,\pi)$ be a saturated Fell bundle with unit fiber $B$ and Picard homomorphism $\psi$, and let $\Gamma$ be a fundamental family of sections of $\pi:\B \to G$.
Then the following assertions hold:
\begin{enumerate}[label={(\roman*)},ref=\ref{lem:autME}.{(\roman*)}]
\item
    Each class in $\Ext(G,B,\psi,[\Gamma])$ can be represented by a Fell bundle of the following form: $\B((B_g,\Psi_{g,h})_{g,h \in G})$ (see Corollary~\ref{cor:fs}).
\item
    Any other Fell bundle $\B((B_g,\Psi'_{g,h})_{g,h \in G})$ representing a class in $\Ext(G,B,\psi,[\Gamma])$ satisfies $\Psi'_{g,h} = \omega(g,h) \Psi_{g,h} := \varphi_{B_{gh}}(\omega(g,h)) \Psi_{g,h}$ (see Lemma~\ref{autME}), for $g,h \in G$, where $\omega \in Z^2_c(G,UZ(B))_{\psi}$.
\item
    The Fell bundles $\B((B_g,\Psi_{g,h})_{g,h \in G})$ and $\B((B_g,\omega(g,h)\Psi_{g,h})_{g,h \in G})$ are equivalent if and only if $\omega \in B^2_c(G,UZ(B))_{\psi}$.
\end{enumerate}
\end{lemma}

\begin{cor}\label{cor:class_top}
Let $G$ be a locally compact group, let $B$ be a unital C\Star algebra, and let $\psi:G \to \Pic(B)$ be a topological Picard homomorphism.
Moreover, let $(\B,G,\pi)$ be a saturated Fell bundle with unit fiber $B$ and Picard homomorphism $\psi$, and let $\Gamma$ be a fundamental family of sections of $\pi:\B \to G$.
Then the map
\begin{gather*}
	H^2_c(G,UZ(B))_{\psi} \times \Ext(G,B,\psi,[\Gamma]) \to \Ext(G,B,\psi,[\Gamma]),
    \\
	\left([\omega],[\B((B_g,\Psi_{g,h})_{g,h \in G})]\right) \mapsto 
    [\B((B_g,\omega \Psi_{g,h})_{g,h \in G})]
\end{gather*}
is a well-defined simply transitive action.
\end{cor}

\begin{example}
Let $G$ be a locally compact group and let $(\B,G,\pi)$ be a Fell line bundle (see Example~\ref{ex:flb} and Example~\ref{ex:ce}).
By Corollary~\ref{cor:class_top}, the equivalence classes of saturated Fell bundles that are topologically equivalent to $(\B,G,\pi)$ are parametrized by $H^2_c(G, \mathbb{C})$.
For instance, if $G = \mathbb{R}$, then it is well-known that $H^2_c(\mathbb{R},\mathbb{T}) \cong \{0\}$.
In contrast, if $G = \mathbb{R}^2$, then $H^2_c(\mathbb{R}^2,\mathbb{T}) \cong \mathbb{R}$ (see, \eg,~\cite[Rem.~2.16]{Neeb07}).
\end{example}

\section{Crossed product bundles}
\label{sec:out}

In this section, we examine a simple yet intriguing class of Fell bundles:

\begin{defi}[\cf~{\cite[Chap.~VIII, Sec.~2.9]{Fell2}}]
A Fell bundle $(\B,G,\pi)$ with unital unit fiber is called a \emph{crossed product bundle} if each fiber $B_g$, for $g \in G$, contains a unitary element, \ie, an element $u_g \in B_g$ satisfying $u_g u_g^* = u_g^* u_g = 1_{B_e}$.

A, possibly noncontinuous, section $u:G \to \B$ of a crossed product bundle $(\B,G,\pi)$ is called \emph{unitary} if each $u_g := u(g) \in B_g$, for $g \in G$, is unitary.
\end{defi}

\begin{lemma}\label{lem:cross}
Let $(\B,G,\pi)$ be a crossed product bundle.
Then the following assertions hold:
\begin{enumerate}[label={(\roman*)},ref=\ref{lem:cross}.{(\roman*)}]
\item\label{cross_one}
    For each $g \in G$, $B_g = u_g B_e = B_e u_g$.
\item\label{cross_sat}
    For all $g,h \in G$, $B_g B_h = B_{gh}$.
    In particular, $(\B,G,\pi)$ is saturated.
\item\label{cross_out}
    The associated Picard homomorphism $\psi_\B : G \to Pic(B_e)$ takes values in $\Out(B_e)$.
\end{enumerate}
\end{lemma}
\begin{proof}
\begin{enumerate}[label={(\roman*)}]
\item 
    Let $g \in G$ and let $u_g \in B_g$ be a unitary element.
    Clearly, $u_g B_e \subseteq B_g$ and $B_e u_g \subseteq B_g$.
    For the reverse inclusion, let $x \in B_g$. 
    Then $x = u_g u_g^* x \in u_g B_e$ and $x = x u_g^* u_g \in B_e u_g$.
\item
    Let $g,h \in G$, and let $u_g \in B_g$ and $u_h \in B_h$ be unitary elements.
    Item~\ref{cross_one} implies that $B_{gh} = u_{gh} B_e =  u_{gh} u_h^* u_g^* u_g u_h B_e \subseteq B_e B_g B_h B_e \subseteq B_g B_h \subseteq B_{gh}$.
\item
    Let $g \in G$, let $u_g \in B_g$ be a unitary element, and let $S_g \in \Aut(B_e)$ be defined by $S_g(b) := u_g b u_g^*$.
    A straightforward computation shows that the map $\Psi_g : M_{S_g} \to B_g$, defined by $\Psi_g(b) := bu_g$, is an isomorphism of Morita equivalence $B_e$-bimodules.
    Consequently, $[B_g] = [M_{S_g}] \in \Out(B)$ (see Section~\ref{sec:pic}).
\end{enumerate}
\end{proof}

\begin{lemma}
\label{lem:out}
Let $(\B,G,\pi)$ be a saturated Fell bundle with unital unit fiber such that its associated Picard homomorphism $\psi_\B : G \to Pic(B_e)$ takes values in $\Out(B_e)$.
Then $(\B,G,\pi)$ is a crossed product bundle.
\end{lemma}
\begin{proof}
Let $g \in G$.
By assumption, $[B_g] \in \Out(B_e)$, meaning there exists $S(g) \in \Aut(B_e)$ and an isomorphism $\Psi_g: M_{S(g)} \to B_g$ of Morita equivalence $B_e$-bimodules (see Section~\ref{sec:pic}).
Define $v_g := \Psi_g(1_{B_e})$, so that $B_g = v_g B_e = B_e v_g$.
This gives $B_e = v_g B_e v_g^*$ and $B_e = v_g^* B_e v_g$ when combined with Remark~\ref{rem:fb:unital}.
In particular, $v_g$ is invertible.
Let $b := v_g^{-1} (v_g^{-1})^*$, and define $u_g := v_g \vert b \vert$, where $\vert b \vert := \sqrt{b^*b}$.
Clearly, $u_g \in B_g$.
Moreover, $u_g u_g^* = v_g b v_g^* = 1_{B_e}$ and $u_g^* u_g = \vert b \vert v_g^* v_g \vert b \vert = \vert b \vert b^{-1} \vert b \vert = 1_{B_e}$.
That is, $u_g \in B_g$ is unitary, which proves the claim.   
\end{proof}

\begin{example}
By Lemma~\ref{lem:out}, every Fell line bundle is a crossed product bundle.
\end{example}

We proceed by considering a crossed product bundle $(\B,G,\pi)$ and 
choosing for each $g \in G$ a unitary element $u_g \in B_g$, with $u_e = 1_{B_e}$.
The maps $S: G \to \Aut(B_e)$, defined by $S(g)(b) := u_g b u_g^*$, and $ \omega: G \times G \to UB_e$, defined by $ \omega(g,h) := u_g u_h u_{gh}^*$,
determine an element $(S,\omega) \in C^1(G,\Aut(B)) \times C^2(G,UB)$ satisfying the conditions
\begin{enumerate}[label=(C\arabic*)]
\item\label{cond:cm1}
    $S(g)(S(h)(b)) \omega(gh) = \omega(g,h) S(gh)(b)$ for all $g,h\in G$, and $b \in B_e$,
\item\label{cond:cm2}
    $S(g) (\omega(h,k)) \omega(g,hk) = \omega(gh,k) \omega(g,h)$ for all $g,h,k \in G$.
\end{enumerate}
Condition~\ref{cond:cm1} is called the \emph{twisted action} condition, while Condition~\ref{cond:cm2} is known as the \emph{twisted cocycle} condition.
It is easy to check that
\begin{align*}
   m_{g,h}(bu_g \otimes_{B_e} cu_h) = b S(g)(c) \omega(gh) u_{gh}
\end{align*}
for all $g,h \in G$ and $b,c \in B_e$.
In order to account for topological considerations, we introduce the following notion:


\begin{defi}\label{def:contcochains}
Let $G$ be a locally compact group and let $B$ be a unital C\Star algebra.
\begin{enumerate}[label=(CC\arabic*)]
\item 
    $C^1_c(G,\Aut(B))$ denotes the set of strongly continuous 1-cochains $G \to \Aut(B)$.
\item 
    $C^2_c(G,UB)$ denotes the set of continuous 2-cochains $G \times G \to UB$.
\end{enumerate}
\end{defi}



Next, let $G$ be a locally compact group, let $B$ be a unital C\Star algebra, let $(\B,G,\pi)$ be the trivial Banach bundle over $G$ whose constant fiber is the Banach space underlying $B$, and let $(S,\omega) \in C^1_c(G,\Aut(B)) \times C^2_c(G,UB)$ satisfy Conditions~\ref{cond:cm1} and~\ref{cond:cm2}.
We define a multiplication $m: \B \times \B \to \B$ and an involution ${}^*:\B \to \B$ as follows:
\begin{gather}
    m((g,b),(h,c)) := (gh,b S(g)(c) \omega(g,h))
    \label{eq:multi_cp}
    \\
    (g,b)^* := (g^{-1},S(g^{-1})(b^* \omega(g^{-1},g)^*)) = (g^{-1},\omega(g^{-1},g)^* S(g^{-1})(b^*))
    \label{eq:invo}
\end{gather}
for all $b,c \in B$ and $g,h \in G$.
Notably, for each $g \in G$, $B_g \cong M_{S(g)}$ as a Morita equivalences $B$-bimodules (see~Section~\ref{sec:pic}).
Using the assumed properties, it is straightforward to show that $(\B,G,\pi)$ together with the above multiplication and involution is a crossed product bundle.
For instance, the multiplication is associative due to  Conditions~\ref{cond:cm1} and~\ref{cond:cm2}, while its continuity follows from the fact that $(S,\omega) \in C^1_c(G,\Aut(B)) \times C^2_c(G,UB)$.
The remaining conditions are verified similarly to those in Section~\ref{sec:construction}.
To document this:

\begin{teo}
\label{teo:cp}
Let $G$ be a locally compact group, let $B$ be a unital C\Star algebra, let $(\B,G,\pi)$ be the trivial Banach bundle over $G$ whose constant fiber is the Banach space underlying~$B$, and let $(S,\omega) \in C^1_c(G,\Aut(B)) \times C^2_c(G,UB)$ satisfy Conditions~\ref{cond:cm1} and~\ref{cond:cm2}.
Then $(\B,G,\pi)$ together with the multiplication and involution defined by~\eqref{eq:multi_cp} and~\eqref{eq:invo}, respectively, is a crossed product bundle, denoted by $(\B,G,\pi,(S,\omega))$.
\end{teo}

\begin{rem}
Suppose we are in the setting of Theorem~\ref{teo:cp}.
\begin{enumerate}[label=(\arabic*)]
\item 
    The topology on $(\B,G,\pi)$ coincides with the topology induced by the \Star algebra of compactly supported sections of $\pi:\B \to G$, which can be intentionally conflated with the topology of compactly supported functions $f: G \to B$ (see Remark~\ref{rem:fb_top}).
\item 
    The crossed product bundle $(\B,G,\pi)$ has continuous unitary sections.
\end{enumerate}
\end{rem}

Let $(\B, G, \pi)$ be a crossed product bundle, and suppose that $u:G \to \B$ is a continuous section thereof. 
Define $\Psi_u$ as the map from $\B$ to the trivial Fell bundle
$G \times B$, defined for $s \in B_g$, for $g \in G$, by $\Psi_u(s) := (g,u_g^* s) = (g,\langle u_g,s \rangle_{B_e})$.
By Lemma~\ref{lem:bb_cont}, Remark~\ref{rem:fb:innerprod}, and the continuity of $u$, $\Psi_u$ is an equivalence between the underlying Banach bundles.
We thus record:

\begin{cor}
If $(\B, G, \pi)$ be a crossed product bundle with continuous unitary sections, then it is equivalent to $(\B,G,\pi,(S,\omega))$ for some $(S,\omega) \in C^1_c(G,\Aut(B)) \times C^2_c(G,UB)$ satisfying Conditions~\ref{cond:cm1} and~\ref{cond:cm2}.
\end{cor}


\section*{Acknowledgement}

This research was partially supported by the University of Warsaw Thematic Research Programme ``Quantum Symmetries''. 
In particular, the authors are grateful to Piotr Hajac for his exceptional hospitality and support.

\appendix

\bibliographystyle{abbrv}
\bibliography{short,RS}

\end{document}